\documentclass[11pt]{amsart}

\usepackage{amssymb,amsmath,amsthm,amscd,mathrsfs}

\newtheorem{teo}{Theorem}[section]
\newtheorem{pro}[teo]{Proposition}
\newtheorem{lem}[teo]{Lemma}
\newtheorem{cor}[teo]{Corollary}
\newtheorem{con}[teo]{Conjecture}

\theoremstyle{definition}

\newtheorem{que}[teo]{Question}

\theoremstyle{remark}
\newtheorem{rem}[teo]{Remark}

\numberwithin{equation}{section}

\begin{document}
\bibliographystyle{amsplain}

\title[Singularities of theta divisors]{Singularities of theta divisors in algebraic geometry}
\author[Casalaina-Martin]{Sebastian Casalaina-Martin }
\thanks{The author was partially supported by NSF MSPRF grant DMS-0503228}
\address{Sebastian Casalaina-Martin, Harvard University, Department of Mathematics,  One Oxford Street, Cambridge MA 02138, USA}
\email{casa@math.harvard.edu}

\subjclass{14K25, 14H42, 14H40}

\dedicatory{Dedicated to Roy Smith}

\date{\today}
\begin{abstract} 
The singularities of
theta divisors have played an important role in the study of algebraic varieties.
These notes survey some of the recent
progress in this subject, using as motivation some well known results, especially those for Jacobians.
\end{abstract}
\maketitle

\section*{Introduction}

The geometric properties of theta divisors have been studied extensively since Riemann originally investigated the vanishing  of theta functions on Jacobians of curves.   Singularities of the theta divisor have proven to be useful for investigating  a number of topics including the injectivity of period maps (e.g. Green \cite{green}, Debarre \cite{dtorelli} and Mumford \cite{mprym}), the characterization of special principally polarized abelian varieties (e.g. Andreotti-Mayer \cite{am} and  Debarre \cite{dprym}), the Kodaira dimension of the moduli space of principally polarized abelian varieties (Mumford \cite{mka}), as well as the irrationality of certain threefolds (e.g. Clemens-Griffiths \cite{cg} and Beauville \cite{bint}).  
The best understood theta divisors are those of Jacobians, where the singularities have been studied extensively, and have been used to give at least partial solutions to these types of problems.  Many of  the results for Jacobians have analogues for other principally polarized abelian varieties (ppavs) such as Prym varieties, and intermediate Jacobians, though the results in these cases are generally less complete.  

There are other geometric properties of theta divisors which have proven to be extremely useful as well:  Gauss maps of theta divisors (e.g. \cite{andreo}), intersections of translates of theta divisors (e.g. \cite{wtri} and \cite{ppr1,ppc}), and  subvarieties of ppavs with minimal class (e.g. \cite{mat} and \cite{ran}), to name a few.  Despite the richness of these notions, in order to keep the notes relatively concise we have chosen to focus on results that are directly related to singularities.  The notes are organized as follows.   After recalling the main results for Jacobians, we survey some of the recent progress in extending these to Prym varieties and intermediate Jacobians of cubic threefolds,  as well as more generally to  those ppavs with a singular theta divisor.  The last section briefly discusses singularities of  generalized theta divisors on moduli spaces of vector bundles over curves.  Although not technically  falling under the category of theta divisors on ppavs, the statements are similar in nature to those for Jacobians, and it seemed appropriate to consider them here.

\subsection{Notation and conventions}  We will work over $\mathbb C$.  We denote by $\mathcal A_g$ the moduli space of principally polarized abelian varieties (ppavs) of dimension $g$,   by $\mathscr M_g$ the  moduli space of smooth curves of genus $g$, and by $\mathcal R_g$ the moduli space of connected \'etale double covers of smooth curves of genus $g$.  Let $\mathscr J:\mathscr M_g \to \mathcal A_g$ be the Torelli map taking a curve to its Jacobian, and denote by $J_g$ the image, and $\bar{J}_g$ the closure of $J_g$.  Let $\mathscr P:\mathcal R_{g+1}\to \mathcal A_g$ be the Prym map taking a double cover to its associated Prym variety, and denote by $P_g$ the image, and  $\bar{P}_g$ the closure.

\section{Jacobians}
The  most well known and best understood abelian varieties are Jacobians.  Recall that given a smooth curve $C$ of genus $g$, we define the Jacobian $JC:=H^0(C,\omega_C)^\vee/H_1(C,\mathbb Z)$.  There is a canonical principal polarization $\Theta$ induced by the intersection pairing on homology.  In this section we recall some of the basic facts about the singularities of the theta divisor in order to motivate the discussion in what follows.

\subsection{The Riemann Singularity Theorem, and Kempf's generalization} \ \ \ 
 The Abel-Jacobi theorem gives a canonical  isomorphism $JC\cong \operatorname{\operatorname{Pic}}^0(C)$, which induces an isomorphism $JC\cong \operatorname{\operatorname{Pic}}^{g-1}(C)$.  Under this identification, Riemann's theorem states that $\Theta$ is a translate of 
$$W_{g-1}:=\{L\in \operatorname{\operatorname{Pic}}^{g-1}(C):h^0(L)>0\},$$  and moreover, the singularities of the theta divisor correspond to line bundles with at least two linearly independent global sections.  More precisely, the Riemann Singularity Theorem (RST) states that
after translation, so that $\Theta$ is identified with $W_{g-1}$,  $\operatorname{mult}_L\Theta=h^0(L)$.
Recall that for a divisor $D$ on a smooth variety $X$ defined locally near a point $x\in X$ by the function $f$, the multiplicity of $D$ at $x$, denoted  $\operatorname{mult}_x D$ is the order of vanishing of $f$ at $x$; equivalently, the multiplicity is the degree of the tangent cone $C_xD$ in the projectivization of the tangent space $\mathbb PT_xX$.

More detailed information about the singularity is given by Kempf's generalization, which we only state in the special case of the theta divisor.  For more details, we refer the reader to Kempf \cite{k} and Arbarello et al. \cite[Chapter VI]{acgh}.  To state the result, we introduce some notation.  To begin, given $L\in \operatorname{\operatorname{Pic}}^{g-1}(C)$, we may identify the tangent space to $\operatorname{\operatorname{Pic}}^{g-1}(C)$ at $L$ with $H^0(K_C)^\vee$.  There is the Petri map
$$
\mu:H^0(L)\otimes H^0(K_C\otimes L^\vee)\to H^0(K_C),
$$
and given bases $x_1,\ldots,x_{h^0(L)}$ and
$y_1,\ldots,y_{h^0(L)}$ of $H^0(L)$ and $H^0(K_C\otimes L^\vee)$ respectively, set $x_iy_j=\mu(x_i\otimes y_j)$.

\begin{teo}[Kempf \cite{k}]  \ \ \ \ \ 
The projective tangent cone $C_L\Theta$ is a Cohen-Macaulay, reduced, normal hypersurface in $\mathbb PH^0(K_C)^\vee$ of degree $h^0(L)$, whose ideal is generated by the determinant of the matrix $(x_iy_j)$ and is supported on 
$$
\bigcup_{D\in |L|} \overline{\phi_K(D)},
$$
where $\phi_K$ is the canonical map of $C$, and $\overline{\phi_K(D)}$ is the span of the points of $D$ in $\mathbb PH^0(K_C)^\vee$.

\end{teo}

The basic idea of the proof is as follows:  setting $C_{g-1}$ to be the symmetric product of the curve, one considers the Abel map
$C_{g-1}\to \operatorname{\operatorname{Pic}}^{g-1}(C)$
which gives a desingularization of the theta divisor.  By Abel's theorem, the fibers are projective spaces, which one checks are in fact smooth as subschemes of $C_{g-1}$.  A careful study of the normal bundle to the fiber over $L$ yields the result.  The key point is a result of Kempf's \cite{k} which states that in this situation  the pushforward of the normal bundle is the tangent cone to $\Theta$ at $L$.  
We also point out that  from this description it follows that $\Theta$ has rational singularities.

\begin{rem}
An immediate consequence of Kempf's theorem is a bound on the rank of the quadric tangent cones to the theta divisor.  Specifically, if $L\in \operatorname{Pic}^{g-1}(C)$ is such that 
$h^0(L)=2$, then Kempf's theorem states that the tangent cone to $\Theta$ at $L$ is 
the determinant of a two by two matrix with linear entries.  Thus the rank of the quadric is at most four, and one can check that it is equal to three exactly when $L$ is a theta characteristic.
\end{rem}

\subsection{Secants to the canonical curve, tangent cones to the theta divisor, and the constructive Torelli theorem}

As in the previous section, let $\phi_K:C\to \mathbb PH^0(K_C)^\vee$ be the canonical map.  Since $C_L\Theta$ sits naturally in $\mathbb PH^0(K_C)^\vee$, the canonical curve, the secant varieties to the canonical curve, and the tangent cones to the theta divisor lie in the same projective space.  
We will use the convention that the $k$-secant variety to a variety $X\subseteq \mathbb P^n$ is the closure of the variety of $k$-dimensional planes meeting $X$ in $k+1$ points.
A proof of the following well known statement is given in Arbarello et al. \cite[Theorem VI 1.6 (i)]{acgh}.
Although the proof is elementary, we give an alternate proof here which is short, and generalizes well to other situations (cf. \cite{casa}):

\begin{pro}
The $k$-secant variety to the canonical curve is contained in $C_L\Theta$ if and only if $h^0(L)\ge k+2$.
     \end{pro}
\begin{proof} 
We will use the basic fact that
for nonzero $\alpha\in H^1(\mathscr O_C)$, 
$[\alpha]\subseteq C_L\Theta$ if and only if
 the cup product map
$\cup \alpha :H^0(L)\to H^1(L)$ has a nontrivial kernel (see for example Arbarello et al. \cite{acgh} or \cite{cmf}).  Now let $p\in C$, and consider the exact sequence
$$
0\to \mathscr O_C \to \mathscr O_C(p)\to \mathscr O_p(p) \to 0.
$$
This induces a coboundary map $\delta:H^0(\mathscr O_p(p))\to H^1(\mathscr O_C)$, and after identifying
$H^1(\mathscr O_C)=H^0(K_C)^\vee$, one can check that for nonzero $t\in H^0(\mathscr O_p(p))$, 
$[\delta(t)]=\phi_K(p)\in \mathbb PH^0(K_C)^\vee$.
On the other hand, after tensoring the above exact sequence by $L$, there is an induced coboundary map
$\partial:H^0(L(p)|_p)\to H^1(L)$, and a local computation shows that  $\cup \delta(t):H^0(L)\to H^1(L)$ is given by the composition 
$$
H^0(L)\stackrel{\cup t}{\to}H^0(L(p)|_p)\stackrel{
\partial}{\to}H^1(L).
$$
For a general point $p\in C$, $h^0(L)=h^0(L(p))$, and so $\partial$ is injective.  Therefore, a general point $p\in \phi_K(C)$ is contained in $C_L\Theta$ if and only if $h^0(L(-p))\ne 0$.  This is the case if and only if $h^0(L)>1$.  In other words, the canonical model of the curve is contained in the tangent cone to $\Theta$ at $L$ if and only if $h^0(L)\ge 2$.  The proof of the general case is similar.
\end{proof}

In particular, the canonical curve is contained in every quadric tangent cone; note that this shows that the rank of such a quadric must be at least three, since the canonical model is non-degenerate.    Recall that Petri's analysis of the canonical ideal shows that it is generated by  quadrics, so long as $g(C)\ge 5$, and $C$ is neither trigonal, nor a plane quintic.  
A theorem of Green's  shows that in fact the degree two part of the canonical ideal is generated by the quadric tangent cones to the theta divisor.  To be precise, let $I_{2,K}$ be the ideal generated by the quadrics containing the canonical model of $C$,  and let $I_2(\Theta)$ be the ideal generated by the quadric tangent cones to the theta divisor.

\begin{teo}[Green \cite{green}]
For a non-hyperelliptic curve $C$ of genus $g(C)\ge 4$, $I_2(\Theta)=I_{2,K}$.
\end{teo}

In other words, a curve of genus at least five, which is neither trigonal nor a plane quintic, is cut out scheme theoretically by the (rank four) tangent cones to the double points of the theta divisor of its Jacobian.  In particular, it can be recovered from its Jacobian.
 This is one version of what is known as the constructive Torelli theorem.  The fact that all curves are uniquely determined by their Jacobian is due to Torelli.  A complete  constructive proof was given by Andreotti \cite{andreo} via the branch locus of the Gauss map of the theta divisor.  

\begin{teo}[Torelli] If $(JC,\Theta_C)\cong (JC',\Theta_{C'})$, then $C\cong C'$. In other words the period map $\mathscr M_g\to \mathcal  A_g$ taking a curve to its Jacobian, is injective.
\end{teo}

In what follows, we will describe several approaches to the injectivity of period maps via singularities of theta divisors.

\subsection{Martens' theorem, and the geometric Schottky problem}
For a variety $X$, define the $k$-th singular locus
$$
\operatorname{Sing}_k X=\{x\in X:\operatorname{mult}_x X\ge k\}.
$$
From the RST it follows that on a Jacobian
$$
\operatorname{Sing}_k \Theta=W^{k-1}_{g-1}:=\{L\in \operatorname{Pic}^{g-1}(C):h^0(L)\ge k\}.
$$
More generally, we can consider the Brill-Noether loci 
$$
W_d^r=\{L\in \operatorname{Pic}^{d}(C):h^0(L)\ge r+1\}.
$$  
A theorem of Martens' \cite{martens} states that if $C$ is a smooth curve of genus $g\ge 3$, $d$ is an integer such that
$2\le d\le g-1$, and $r$ is an integer such that
$0<2r\le d$, then $ \dim W^r_d\le d-2r$,
and equality holds if and only if $C$ is hyperelliptic.  
It follows that
\begin{equation}\label{eqnmart}
\dim \operatorname{Sing}_k\Theta \le g-2k+1,
\end{equation}
 and equality holds if and only if $C$ is hyperelliptic.  The $W^r_d$ may be viewed as degeneracy loci of maps of vector bundles, and from this perspective, for instance, one can show that $\dim \operatorname{Sing} \Theta\ge g-4$.  Thus
for a Jacobian of dimension at least four,
the dimension of the singular locus of the theta divisor is exactly $g-4$, unless the curve is hyperelliptic, in which case the dimension is $g-3$.

The Schottky problem in the most general sense is to characterize Jacobians among all ppavs of dimension $g$.  The literature on this subject is extensive, going back to Schottky and Jung's original papers \cite{schottky, scj}, and complete solutions have been provided by
Arbarello-de Concini \cite{ad1, ad2}, Mulase \cite{mul}, Shiota \cite{shiota}, and most recently by Krichever \cite{kr}, who proved the celebrated tri-secant conjecture of Welters \cite{wcriterion}.  In these notes
we will focus on Andreotti and Mayer's approach to this problem via the singularities of the theta divisor.
Using the description above as motivation, Andreotti and Mayer defined subloci of $\mathcal  A_g$ 
$$
N_{g,\ell}=\{(A,\Theta)\in \mathcal  A_g: \dim \operatorname{Sing} \Theta\ge \ell\}.
$$
If the space $\mathcal A_g$ is clear from the context, we will write $N_\ell$ for $N_{g,\ell}$.
Andreotti and Mayer proved that $N_0\ne \mathcal  A_g$, as well as the following theorem: 
\begin{teo}[Andreotti-Mayer \cite{am}] For $g\ge 4$, 
$\bar{J}_g$ is an irreducible component of 
$N_{g-4}$. 
\end{teo}

 Although a theorem of Beauville's \cite{bschot} shows that for all $g\ge 4$, $N_{g-4}\ne \bar{J}_g$, this approach to the Schottky problem has motivated a number of efforts to classify other ppavs via the singularities of the theta divisor.
We will also discuss later (see Section \ref{secaml}) a modification of Andreotti and Mayer's approach to the Schottky problem proposed by Ciliberto-van der Geer.

\section{Prym varieties}  Prym varieties are another class of abelian varieties which can be studied using curves.   All of the results for Jacobians outlined above have analogues for Prym varieties, although many of these hold only for generic curves.  There is an ongoing effort to determine exactly under which conditions these generalizations hold.
Recall that given $\pi:\widetilde{C}\to C$ a connected \'etale double cover of a smooth curve $C$ of genus $g$, there is an induced norm map
$$\operatorname{Nm}:\operatorname{Pic}(\widetilde{C})\to \operatorname{Pic}(C),$$  and we define the Prym variety $P:=(\ker \operatorname{Nm})_0\subset J\widetilde{C}$ to be 
the connected component of the kernel of the norm.  The principal polarization $\widetilde{\Theta}$ on $J\widetilde{C}$ restricts to give twice a principal polarization $\Xi$ on $P$.  We call $(P,\Xi)$ the principally polarized Prym variety associated to the double cover $\pi:\widetilde{C}\to C$; it is $g-1$ dimensional.
In what follows we will also want to fix the following notation.
Associated to such a double cover is an involution
$\tau:\widetilde{C}\to \widetilde{C}$ given by exchanging the sheets of the cover, and also a line bundle $\eta\in \operatorname{Pic}^0(C)$ such that $\eta\otimes \eta \cong \mathscr O_C$.

\subsection{The Riemann Singularity Theorem for Pryms, and consequences of Kempf's theorem}
The question of extending the RST to Pryms goes back to Mumford's original paper \cite{mprym}.  To be more precise, by extending the RST we mean relating the multiplicity of a point $x\in \Xi$ to the space of global sections of a line bundle on $\widetilde{C}$ or $C$.  As a parallel to Riemann's theorem,   Mumford showed that
as a set
$$
P=\{L\in \operatorname{Pic}^{2g-2}(\widetilde{C}):\operatorname{Nm} L\cong \omega_C, \ h^0(L)\equiv 0 \pmod{2}\}
$$
and 
$$
\Xi=\{L\in P:h^0(L)>0\}.
$$
Since $\widetilde{\Theta}|_P=2\Xi$, it follows that 
$\operatorname{mult}_L\Xi\ge h^0(L)/2$ with equality holding if and only if $T_LP\nsubseteq C_L\widetilde{\Theta}$. 
In other words, there is immediately an RST for Pryms at all points such that $T_LP\nsubseteq C_L\widetilde{\Theta}$.  
This motivates the following definition.  We will say that a point $L\in \Xi$ is RST-stable if 
$T_LP\nsubseteq C_L\widetilde{\Theta}$, and  RST-exceptional otherwise.  Note that this is slightly different from the standard notation in the literature;  in the standard  notation $L$ is exceptional  if  it is  RST-exceptional \emph{and} $h^0(L)=2$, and it is stable  otherwise.
We will see that we will want to have this stronger distinction later.   

Using a strengthened version of Martens' theorem, Mumford showed:

\begin{teo}[Mumford \cite{mprym}]
If $g(C)\ge 5$, and $C$ is not hyperelliptic, then $\dim \operatorname{Sing}\Xi \le g-5$.
\end{teo}

 For $g\ge 5$, a complete  list was given of double covers for which $\dim \left(\operatorname{Sing}\Xi \right)$ $= g-5$, and it was shown that when $C$ is hyperelliptic the Prym is a hyperelliptic Jacobian or the product of hyperelliptic Jacobians.  Note also that for $g(C)\le 4$, the Prym is a Jacobian (or product of Jacobians) and so the dimension of the singular locus is known in this case as well.  We refer the reader to the paper for more details.
A key step in the proof is to give a description of those line bundles which are exceptional in the standard notation; i.e.  $L$ is RST-exceptional and $h^0(L)=2$.
A more general statement for all RST-exceptional singularities was  proven recently by Smith and Varley  using methods similar to (though a great deal more involved than) those used in Mumford's original proof for the case $h^0(L)=2$.

\begin{teo}[Smith-Varley \cite{svriemann}]\label{teosvrst}
Suppose $L\in \Xi$.  Then $L$ is RST-exceptional, i.e.  $\operatorname{mult}_L\Xi>h^0(L)/2$, if and only if $L\cong \pi^*M \otimes \mathscr O_{\widetilde{C}}(B)$, where $M$ is a line bundle on $C$ with $h^0(M)>h^0(L)/2$, and $B\ge 0$ is an effective divisor on $\widetilde{C}$ such that
$\tau^*B \cap B=\emptyset$. 
\end{teo}

This theorem gives a necessary and sufficient condition for determining when $\operatorname{mult}_L\Xi =h^0(L)/2$, but does not address the issue of the multiplicity of the points at which $\operatorname{mult}_L\Xi >h^0(L)/2$.  
A solution to this problem, and thus a completion of the RST for Pryms, was given by the author.
The proof is based on the detailed analysis of second order deformations of $L$.  The methods also yield a proof of the previous theorem. 
\begin{teo}[\cite{casa}]
If $L\cong \pi^*M \otimes \mathscr O_{\widetilde{C}}(B)$, where $M$ is a line bundle on $C$ with $h^0(M)>h^0(L)/2$, and $B\ge 0$ is an effective divisor on $\widetilde{C}$ such that
$\tau^*B \cap B=\emptyset$, then $\operatorname{mult}_L\Xi=h^0(M)$.
\end{teo}

Together with Mumford's strengthened version of Martens' theorem, one can use the previous theorems to  give a bound on the dimension of $\operatorname{Sing}_k\Xi$:

\begin{cor}[\cite{casa}]\label{corprym} For a Prym variety $(P,\Xi)$ of dimension $g-1$, 
$$\dim \operatorname{Sing}_k\Xi \le (g-1)-2k+1.$$
\end{cor}

A result of Beauville's states that in the case that  $k=2$, equality holds only if $C$ is hyperelliptic.  
It would be interesting to know for other $k$ to what extent equality holding implies that $C$ is hyperelliptic.  
At least 
for $g=6$, and $k=3$, equality holds both for hyperelliptic Jacobians, and  intermediate Jacobians of the cubic threefolds.
In fact for all ppavs of dimension four and five it is possible to determine exactly when equality holds.  We refer the reader to the Tables 1 and 2.  We pose the following question for higher genera. 
\begin{que}
For $g\ge 7$ and $k\ge 3$, what are the Prym varieties $(P,\Xi)$ of dimension $g-1$ such that $\dim \operatorname{Sing}_k\Xi = (g-1)-2k+1$?  
\end{que}

We now turn our attention to the question of generalizing  Kempf's theorem to the case of Prym varieties.
In the case of the Jacobian, the theta divisor could be viewed as the zero locus of the determinant of a map of vector bundles.  For the Prym variety, results of Smith and Varley \cite{svpfaf} show that $\Xi$ can be viewed as the zero locus of the  Pfaffian of a map of vector bundles.  In particular, the same can be said of the tangent cones.  We refer the reader to this reference for more details. Although not derived from a global description, it was an observation of Mumford's that in the case of RST-stable singularities (i.e. such that $T_LP\nsubseteq C_L\widetilde{\Theta}$), the restriction of Kempf's matrix gives a Pfaffian description of the tangent cone. A consequence is

\begin{pro}[Mumford \cite{mprym}]
If $L\in \Xi$ is RST-stable, then the tangent cone is the zero locus of the Pfaffian of an $h^0(L)\times h^0(L)$ skew-symmetrc matrix. It follows that
if $L\in \Xi$ is an RST-stable double point (i.e. $mul_x\Xi=h^0(L)/2=2$),
then $\operatorname{rank} C_x\Xi\le 6$.
Moreover, if $g\ge 7$, then $\operatorname{rank} C_x\Xi=6$ if the Prym-Petri map 
$
\bigwedge^2 H^0(L) \rightarrow H^0(\omega_C\otimes \eta)
$
given by $s\wedge t\mapsto \frac{1}{2} (s\cdot \tau^*t-\tau^*s \cdot t)$
is injective.
    \end{pro}
    
A result of Welters' \cite[Theorem 1.11]{wgp} shows that the Prym-Petri map is injective for a general curve.
One can check directly that the existence of an RST-exceptional double point implies that the existence of a line bundle on $C$ for which the Petri map is not injective.  Due to Gieseker's theorem \cite{gies} that this map is injective on a general curve, it follows that for a general Prym variety of dimension at least six, the double points of the theta divisor have rank six.  We refer the reader to \cite{svtorelli,svpac} for more details about the existence of double points with smaller rank.
We cite here the following result giving 
 further details about the structure of the Pfaffian:

\begin{pro}[Smith-Varley 
\cite{svtorelli} Proposition 3.6]\label{propf}
Suppose that $L\in \Xi$ is an RST-stable double point.
If $L=\pi^*M\otimes \mathcal{O}_{\tilde{C}}(B)$, for a line bundle $M$ on $C$ such that
$h^0(C,M)=2$, and an effective divisor $B$ on $\tilde{C}$ such that $B\cap\tau^*B=\emptyset$, then there is a basis such that the upper left 
 $2\times 2$ block of the matrix above is zero.  In particular, the rank of the quadric tangent cone is at most four.   
    \end{pro}

A computation of Beauville's shows the following
\begin{lem}[Beauville \cite{bschot} p.191]
If $L\in \Xi$ is an RST-stable double point, and moreover, $L$ is a base point free theta characteristic, then the rank of the quadric tangent cone is exactly four.
\end{lem}
 
 We also point out that Kempf's theorem implies that at an RST-stable point $L\in \Xi$, as a set,
$$
C_L\Xi=\bigcup_{D\in |L|} \left(
\overline{\phi_{K_{\widetilde C}}(D)} \cap \mathbb PH^0(K_C\otimes \eta)^\vee\right).
$$

It remains to give a good  generalization of Kempf's theorem to the case of RST-exceptional points of $\Xi$.  Ideally this would yield a nice description of the rank of the quadric tangent cone: 

\begin{que}\label{queprym}
Suppose that $L\in \Xi$ is an RST-exceptional double point; i.e. $L\cong \pi^*M \otimes \mathscr O_{\widetilde{C}}(B)$,  where $M$ is a line bundle on $C$ with $h^0(M)=h^0(L)=2$, and $B\ge 0$ is an effective divisor on $\widetilde{C}$ such that
$\tau^*B \cap B=\emptyset$.
What is the rank of $C_L\Xi$?
\end{que}

In \cite{casa} Section 6, the author has given a determinantal description of the tangent cones at such double points, but so far this has not yielded information about the rank.

\subsection{Secants to the Prym-canonical curve, tangent cones to the theta divisor, and the constructive Torelli theorem}
Consider the Prym canonical map $\phi_{\eta}:C\to \mathbb PH^0(K_C\otimes \eta)^\vee$ induced by the linear system
$|K_C\otimes \eta|$.  Call the image the Prym canonical model of $C$.  Since the tangent space to the Prym variety at a point can be identified with
$H^0(K_C\otimes\eta)^\vee$, it follows that the projectivized tangent cones to $\Xi$ lie in the same space as the secant varieties to the Prym canonical model.  It was first shown by Tjurin \cite[Lemma 2.3, p.963]{tjg,tjgc} that if $L$ is an RST-stable double point on $\Xi$, then $\phi_\eta(C)\subseteq \mathbb PC_L\Xi$.

As with Jacobians, one can ask whether the Prym canonical model can be recovered from the  double points of the theta divisor.  For Jacobians the key points were Petri's theorem, and Green's theorem.  Analogues of these theorems for Prym varieties have been proven by a number of authors in different contexts.  We refer the reader to Smith-Varley \cite{svtorelli} for a complete discussion, and state below a generic version of the statement due to Debarre.

\begin{teo}[Debarre \cite{dtorelli}]
For a generic Prym variety of dimension at least seven, the Prym canonical curve is cut out as a scheme by the (rank six) tangent cones to the stable double points of the theta divisor.
\end{teo}

This gives a constructive proof of a special case of the generic Torelli theorem for Pryms:
\begin{teo} [Friedman-Smith \cite{fs}] 
Suppose that  $g\ge 7$. Then the Prym map $\mathscr P:\mathcal R_g\to \mathcal A_{g-1}$ is generically injective.
\end{teo}

For dimension reasons, it is easy to see that the Prym map is not injective for $g\le 6$, and it is well known that the map is dominant in these cases.
That the Prym map is never injective  can be seen from Mumford's  analysis  \cite{mprym} of the Pryms associated to hyperelliptic curves, or Donagi's \cite{dtet} tetragonal construction. 
The question remains to show exactly when the map is injective.  
There are a number of partial results in this direction, and we refer the reader to the survey Smith-Varley \cite{svtorelli} for a complete discussion of the current situation.
We state the main conjecture here for reference; this is due to Donagi \cite[Conjecture 4.1]{dtet}, with  modifications by Verra \cite{verra} and Lange-Sernesi \cite{langeser}:

\begin{con}[Donagi]
Let $U=\{(\widetilde{C},C)\in \mathcal R_g:\operatorname{Cliff}(C)\ge 3\}$.
The restricted Prym map 
$\mathscr P|_U:U\to \mathcal A_{g-1}$ is injective.
\end{con}

Recall the definition of the Clifford index:
$$
\operatorname{Cliff}(C):=\min\{\deg(D)-2\dim|D|: h^0(D)\ge 2,\ h^1(D)\ge 2\}.
$$

We now change direction.
In analogy with the case of Jacobians, and in order to generalize Tjurin's result for RST-stable double points, it is natural to ask for 
the exact relation between the secant varieties of the Prym canonical model, and the tangent cones to $\Xi$.
Smith and Varley observed in \cite[p.241, line 8]{svtorelli} that if $(P,\Xi)$ is a Jacobian, then at a general exceptional double point $L$, $\phi_\eta(C)\nsubseteq C_L\Xi$.  
A complete description of the secant varieties generalizing this observation was given by the author.

\begin{teo}[\cite{casa}]
Suppose that $L\in \operatorname{Sing} \Xi$.  The $k$-secant variety to $\phi_\eta(C)$ is contained in $C_L\Xi$ if and only if $h^0(L)\ge 2(k+2)$.
\end{teo}

In particular, the Prym canonical model is contained in a quadric tangent cone to $\Xi$ if and only if the double point is RST-stable.

\subsection{The geometric Schottky problem}
For dimension reasons, it is easy to see that the closure of the image of the Prym map, $\bar{P}_g$, is a proper subvariety of $\mathcal A_{g-1}$ for all $g\ge 7$, and thus for $g\ge 7$ we can ask the Prym Schottky problem; i.e. to characterize Prym varieties among all ppavs.  We describe here an approach similar in spirit to Andreotti and Mayer's approach to the Schottky problem, using the dimension of the singular locus of the theta divisor to give a partial characterization.

To  motivate the statements, we cite several more results on the dimension of the singular locus of the Prym theta divisor.  In particular, Welters \cite{wgp} and Debarre \cite{dtorelli} have shown that the singular locus of $\Xi$ has dimension at least $g-7$.  Together with work of Debarre \cite{dprym} it follows that for a generic Prym variety, the locus of RST-exceptional singularities is empty, while the locus of RST-stable singularities is irreducible of dimension $g-7$ for $g\ge 8$, reduced of dimension zero for $g=7$, and empty if $g\le 6$.  
We refer the reader to Smith-Varley \cite{svtorelli,svpac} for more details, and further references describing the singular locus.
In analogy with Andreotti and Mayer's approach to the Schottky problem, there is the following theorem of Debarre's.
\begin{teo}[Debarre \cite{dprym}] For $g\ge 8$, the Prym locus $\bar{P}_{g}\subseteq \mathcal  A_{g-1}$
is an irreducible component of $N_{g-7}$.
\end{teo}

Note for $g=7$, the generic Prym variety has a theta divisor which is singular in dimension zero; i.e. $P_7\subseteq N_0\subseteq \mathcal  A_6$.  On the other hand,  $\operatorname{codim} P_7=3$, while $N_0$ is a divisor (see Section \ref{secaml}), and thus $P_7$ is not an irreducible component of $N_0$ in $\mathcal A_6$.
We point out that recently a complete characterization of Prym varieties has been given by Grushevsky and Krichever \cite{gk} in terms of  quadrisecant planes of the associated Kummer variety.

We close this section by describing one more approach to the Schottky problem for Jacobians, which is based on the study of linear series consisting of singular $2\Theta$-divisors, and for which there are some recent results in the case of Prym varieties.
To begin, consider a ppav $(A,\Theta)$ of dimension at least four, with $\Theta$ symmetric.  One sets $\Gamma_{00}\subset H^0(A,2\Theta)$ to be the subspace of sections vanishing to order four at the origin and  $V_{00}$ to be the base locus of the associated linear series.  
Van Geemen and van der Geer \cite{vgvdg} have made the following conjecture, which we phrase following Donagi \cite{dschot}:

\begin{con}[van Geemen and van der Geer \cite{vgvdg}] Let $(A,\Theta)$ be an indecomposable ppav of dimension $g$.
\begin{enumerate}
\item If $(A,\Theta)$ is the Jacobian of a smooth curve $C$, then as a set $V_{00}$ is  equal to the difference surface 
$$C-C:=\{\mathscr O_C(p-q):p,q\in C\}.$$

\item If $(A,\Theta)$ is not a Jacobian,  then as a set $V_{00}=0$.
\end{enumerate}
\end{con}

The first part of the conjecture has been proven by Welters \cite{wcc} and Beauville-Debarre \cite{bdso}, except for the case of non-hyperelliptic curves of genus four, for which the locus $V_{00}$ was worked out explicitly.
(See also
Izadi \cite{izft} for a scheme theoretic description).  The second part of the conjecture has been proven in a number of special cases.  For $g=4$, this  was proven by 
Izadi \cite{izgs}.  Beauville-Debarre-Donagi-van der Geer \cite{bddvdg} prove this for the intermediate Jacobian of a cubic threefold, and the Prym varieties of even double covers of plane curves.   Beauville-Debarre  \cite[Theorem 2]{bdso} prove that for the general ppav of dimension at least four, $\dim V_{00}=0$.  
With respect to a result of  Beauville's  that $\bar{J}_g \subseteq \bar{P}_{g+1}$ for all $g$ (\cite[Theorem 5.4]{bschot} with a Wirtinger cover), as a special case of the conjecture, it would be interesting to know to what extent $\dim V_{00}\ge 2$ characterized Jacobians among Pryms.
In this direction Izadi \cite{izso} proves that for a general Prym variety of dimension $\ge 16$, $\dim V_{00}\le 1$.
We also point out that there is an infinitesimal version of the conjecture stated in \cite{bdso} for which a number of similar results have been proven.

\section{Cubic threefolds}
Cubic threefolds, that is smooth cubic hypersurfaces in $ \mathbb P^4$, have been studied since the nineteenth century.  In the nineteen seventies, Clemens and Griffiths gave a detailed analysis of the canonical principal polarization of the intermediate Jacobian of a cubic threefold.  The singularities of the theta divisor played a central part in their proof of the irrationality of the cubic, as well as in Mumford's proof of Clemens-Griffiths' Torelli theorem for cubic threefolds.  We revisit these well known results, and then state some  recent results including one of Friedman and the author's characterizing intermediate Jacobians via the geometry of their theta divisor.
Recall that for a cubic threefold $X$ the intermediate Jacobian is defined as $JX:=
H^{1,2}(X)/H^3(X,\mathbb Z)$, which is a five dimensional abelian variety.  There is a canonical principal polarization $\Theta_X$, given by the hermitian form $H$ on $H^{1,2}(X)$ defined by
$H(\alpha,\beta)=2i\int_X\alpha\wedge \bar{\beta}$.

\subsection{The Torelli theorem and irrationality of cubic threefolds}
The unirationality of the cubic threefold was known since classical times:  fixing a line $\ell\subset X$, one can consider the $\mathbb P^2$-bundle $P\to\ell$ consisting of the lines tangent to $X$ at a point of $\ell$.  There is a dominant rational map $P\dashrightarrow X$ taking a line $L\in P$, which by definition meets $X$ with multiplicity two at a point of $\ell$, to its residual intersection point with $X$.

In \cite{cg},  Clemens and Griffiths observed that if a threefold is rational, its intermediate Jacobian must be the Jacobian of a curve, or the product of Jacobians of curves.   
  Through a careful study of the Fano surface of lines, and the  branch locus of the Gauss map for the theta divisor, Clemens and Griffiths prove the following.
\begin{teo}[Clemens-Griffiths \cite{cg}] Let $X$ and $X'$ be cubic threefolds.
\begin{enumerate}
\item If $(JX,\Theta_X)\cong (JX',\Theta_{X'})$, then $X\cong X'$. 
\item $(JX,\Theta_X)$ is not isomorphic to the Jacobian of a curve; it follows that cubic threefolds are irrational. 
\end{enumerate}
\end{teo}

A theorem of Mumford's  gives a precise description of the singularities of the theta divisor.

\begin{teo}[Mumford \cite{mprym}]
The intermediate Jacobian $(JX,\Theta_X)$ of a cubic threefold $X$  is the Prym variety associated to the double cover of a smooth plane quintic.  Moreover, $\operatorname{Sing} \Theta_X=\{x\}$, and $C_x\Theta_X\cong X$.
\end{teo}

In particular this establishes a constructive version of the Torelli theorem for cubic threefolds via the singularities of the theta divisor.
Also, since the intermediate Jacobian of a cubic threefold is a five dimensional ppav,  
this gives another proof of the irrationality of the cubic threefold.  Indeed, 
if the cubic were rational, the singular locus of the theta divisor would be of dimension at least one.

Detailed proofs of Mumford's theorem have been given by Beauville \cite{bmumford}, and Clemens \cite{cleminf}.  A related result was proven by Tjurin in \cite{tjurin}.  Recently, a new  proof was given by Smith-Varley \cite{svpac} (see also  \cite{casathesis}).
We point out here that Clemens and Griffiths' observation that the intermediate Jacobian of a rational threefold is isomorphic to the product of Jacobians of curves has been used to show the irrationality of other threefolds.  We direct the reader for instance to \cite{bint}, where  the rationality of conic bundles is discussed.

\subsection{The Schottky problem for cubic threefolds}
In analogy with the Schottky problem, one can ask for a characterization of intermediate Jacobians of cubic threefolds among all ppavs of dimension five.  In \cite{cmf}, Friedman and the author prove the following statement giving a converse to Mumford's theorem, and a geometric solution to the Schottky problem for cubic threefolds.

\begin{teo}[\cite{cmf}]
Suppose that $(A,\Theta)$ is a five dimensional ppav, with $\operatorname{Sing} \Theta=\{x\}$, and $\operatorname{mult}_x\Theta=3$.  Then $(A,\Theta)$ is the intermediate Jacobian of a smooth cubic threefold. 
\end{teo}

The author has generalized this to show that the locus of ppavs of dimension five whose theta divisor contains a triple point, consists of three ten dimensional irreducible components, one of which is $\bar{I}$, the closure of the locus of intermediate Jacobians.
To be precise, set $N^3_0\subset \mathcal A_5$ to be the locus of ppavs whose theta divisor contains a point of multiplicity three. Let $\theta_{\textnormal{null}}\subset \mathcal  A_4$ be the locus of ppavs with a vanishing theta null, let $\mathcal  A_1\theta_{\textnormal{null}}$ and $\mathcal  A_1\bar{J}_4$, be the loci corresponding to products of the respective ppavs.

\begin{teo}[\cite{casaschot}]\label{teocasa3} In $\mathcal A_5$,
$$
N_0^3=\bar{I}\cup \mathcal  A_1\theta_{\textnormal{null}}\cup \mathcal  A_1\bar{J}_4.
$$
In particular, an indecomposable ppav of dimension five with a triple point on the theta divisor is either the Jacobian of a hyperelliptic curve, or the intermediate Jacobian of a cubic threefold.
\end{teo}

This has an amusing geometric consequence for tangent cones to the theta divisor at triple points.

\begin{pro}
Suppose that $(A,\Theta)\in N^3_0\subset \mathcal A_5$.  For $z\in \operatorname{Sing}_3\Theta$, set $X=C_z\Theta$.  $X$ is irreducible, if and only if one of the following holds
\begin{enumerate}
\item $X$ is a smooth cubic threefold, and $(A,\Theta)\cong (JX,\Theta_X)$;
\item $X$ is the secant variety to the rational normal curve (singular along this curve), and there exists a family of smooth cubic threefolds 
$\mathscr X_t\to \Delta$
 over the unit disc $\Delta\subset \mathbb C$, degenerating to $\mathscr X_0=X$ such that
$(A,\Theta)\cong \lim_{t\to 0}JX_t$.  
In this case $(A,\Theta)$ is the Jacobian of a hyperelliptic curve. 
\end{enumerate}
\end{pro}

\begin{proof}
From the previous theorem, and  Theorem \ref{teocasa5} below which states that there are no indecomposable ppavs of dimension four with a triple point on the theta divisor, one can check that for $z\in \operatorname{Sing}_3\Theta$, 
$C_z\Theta$ is irreducible if and only if 
$(A,\Theta)$ is indecomposable, and thus 
$(A,\Theta)$ is the intermediate Jacobian of a cubic threefold, or the  Jacobian of a hyperelliptic curve.  Case one is Mumford's theorem.  So assume that  $C$ is a smooth hyperelliptic curve of genus five.  To show that $X=C_z\Theta$ is the secant variety to the rational normal curve, one uses Kempf's theorem.  Indeed, $z$ corresponds to the unique line bundle $L$ whose associated complete linear series is a $g^2_4$.   After an appropriate choice of basis for $H^0(L)$, one can see using Kempf's theorem that
there is a basis $Z_0,\ldots,Z_4$ of $H^0(\mathbb P^4,\mathscr O_{\mathbb P^4}(1))$ such that  $C_z\Theta$ is defined by the determinant  of a matrix of the form
\[ \left( \begin{array}{ccc}
Z_0 & Z_1 & Z_2 \\
Z_1 & Z_2 & Z_3 \\
Z_2 & Z_3 & Z_4 \end{array} \right).
\]
Such a determinant defines the secant to the rational normal curve. 
Finally, the existence of the family $\mathscr X\to \Delta$ follows from a result of Collino \cite[Theorem 0.1]{collino} (for a proof of this statement via Prym varieties and degenerations of plane quintics, see \cite{casaschot}).
\end{proof}

 Recently several authors have considered projective compactifications of the moduli space of smooth cubic threefolds,  using both GIT constructions (Allcock \cite{allcock}, and  Yokoyama \cite{yoko}) as well as period maps for certain cubic fourfolds (Allcock-Carlson-Toledo \cite{act} and  Looijenga-Swiersta \cite{ls}).  Since the period map taking a cubic to its intermediate Jacobians is an embedding, the closure of the locus of intermediate Jacobians of cubic threefolds in  
 $\mathcal  A_5^*$, the Satake compactification of $\mathcal  A_5$, yields another projective compactification.   One would be interested in understanding the boundary locus in this compactification, as well as the relation of this space to the other compactifications mentioned above.  
 
In terms of the boundary locus $\partial I$ in $\mathcal A_5$, Theorem \ref{teocasa3} has the following consequence:  if $(A,\Theta)\in \partial I$ is indecomposable, then it is the Jacobian of a hyperelliptic curve, otherwise it is the product $(E,\Theta_E)\times (A,\Theta)$ where $(A,\Theta)$ is either a Jacobian or has a vanishing theta null.  In the case that $A$ is indecomposable, Sam Grushevsky has shown the author an argument using theta functions and results from \cite{gsm0} that $A$ is a Jacobian with a vanishing theta null.   
A complete description of the boundary locus in $\mathcal A^*_5$, as well as an explicit description of the relation between this compactification and the others mentioned above will appear elsewhere, in joint work of Radu Laza and the author. 

The boundary of the intermediate Jacobian locus in the toroidal compactifications can be considered as well.  Using results of Friedman-Smith \cite{fsdeg}, and Alexeev-Birkenhake-Hulek \cite{abh}, it may be possible to give at least a partial description of this locus.  See Gwena \cite{gwena} for some computations along these lines.

\section{Subloci of $\mathcal  A_g$}
In this section we consider subloci of $\mathcal A_g$ defined by conditions on the singularities of the theta divisor.

\subsection{Andreotti-Mayer loci}\label{secaml}

In addition to giving a partial solution to the Schottky problem, the Andreotti-Mayer loci $N_\ell\subset \mathcal  A_g$ have been studied  in a variety of other contexts.  We begin by considering the loci $N_0$.  In \cite{bschot} Beauville showed that in $\mathcal A_4$, $N_0$ is a divisor, and more precisely that $N_0=\bar{J}_4 \cup \theta_{\textnormal{null}}$.  In fact Mumford \cite{mka} observed that the same proof shows that $N_0$ is a divisor for all $g$, and described a scheme structure for this divisor;  it is a result of Debarre's that for all $g\ge 4$, 
$$
N_0=\theta_{\textnormal{null}}+2N_0',
$$
where  $N_0'$ is an irreducible divisor whose generic point corresponds to a ppav whose theta divisor has two singular points.  We note that the generic ppav in $\theta_{\textnormal{null}}$ has a unique singular point, a double point at a two-torsion point of the abelian variety.

This divisor was used by Mumford to analyze the Kodaira dimension of $\mathcal  A_g$.  In \cite{mka}, he constructed a partial compactification $\tilde{\mathcal  A}_g^{(1)}$ of $\mathcal  A_g$, with Picard group  generated  (over $\mathbb Q$) by  the hodge class $\lambda$ and  the class of the boundary $\delta$.  Let $\tilde{N}_0$ denote the closure of $N_0$ in this partial compactification.
\begin{teo}[Mumford \cite{mka}]
  On $\tilde{\mathcal  A}_g^{(1)}$,
\begin{enumerate}
\item $[K_{\tilde{\mathcal  A}_g^{(1)}}]=
(g+1)\lambda-\delta$.

\item $[\tilde{N}_0]=\left(\frac{(g+1)!}{2}+g!\right)\lambda -\frac{(g+1)!}{12}\delta$
\end{enumerate}
As a consequence, $\mathcal  A_g$ is of general type for $g\ge 7$.
\end{teo}

For $g\le 3$, $\mathcal  A_g$ is rational (for $g=3$ see Katsylo \cite{katsylo}), while for $g=4,5$ results of  Clemens \cite{cuni} and Donagi \cite{duni} respectively show that $\mathcal A_g$ is unirational.  
The Kodaira dimension of  $\mathcal  A_6$ is not known.

In contrast to the situation for $N_0$, for $\ell>0$,  the loci $N_\ell$ are not well understood.  One exception is the case  $N_1\subset\mathcal  A_5$, where results of Debarre \cite[Proposition 8.2]{dsur} and Donagi \cite[Theorem 4.15]{dschot} show that $N_1$ consists of five irreducible components, one of which is $\bar{J}_5$, and  another of which is $\mathcal A_{1,4}$, the locus consisting of products of elliptic curves with ppavs of dimension four.  The remaining components which  we will denote by  $A$, $B$, and $C$, have dimensions $10$, $9$, and $9$ respectively.  Recently Ciliberto and van der Geer have made some progress giving bounds on the codimension of the $N_\ell$.  They prove the following theorem
\begin{teo}[Ciliberto-van der Geer \cite{cvdg1}, Theorems 2.1 and 2.2] In the notation above
\begin{enumerate}
\item For $g\ge 4$, and $1\le \ell\le g-3$, $\operatorname{codim} N_{g,\ell}\ge \ell+2$
\item $g\ge 5$, and $g/3\le \ell\le g-3$, $\operatorname{codim} N_{g,\ell}\ge \ell+3$
\end{enumerate}
\end{teo}

They do not expect that these bounds are sharp, and in fact have made the conjecture below, which they have proven in the case $\ell=1$.

\begin{con}[Ciliberto-van der Geer \cite{cvdg2}, Conjecture 1.1] If $1\le \ell\le g-3$ and $N$ is an irreducible component of $N_{g,\ell}$ whose general point corresponds to an abelian variety with endomorphism ring $\mathbb Z$, then $\operatorname{codim}_{\mathcal  A_g}(N)\ge \binom{\ell+2}{2}$.  Moreover, equality holds if and only if one of the following holds:
\begin{enumerate}
\item $g=\ell+3$ and $N=\bar{J}^h_g$;
\item $g=\ell +4$ and $N=\bar{J}_g$.
\end{enumerate}
\end{con}

In other words, a special case of the conjecture is that the Jacobian locus in $\mathcal  A_g$ is exactly the sublocus of the Andreotti-Mayer locus $N_{g-4}$ which is of the correct dimension $3g-3$, and whose general point corresponds to an abelian variety with endomorphism ring $\mathbb Z$.

\subsection{Generalized Andreotti-Mayer loci}
There has also been some recent progress in understanding 
the dimension of $\operatorname{Sing}_k\Theta$ for $k>2$.  We call the loci in $\mathcal  A_g$ parameterizing these ppavs ``generalized Andreotti-Mayer loci'', and use the following notation:
$$
N^k_\ell=\{(A,\Theta)\in \mathcal  A_g:
\dim \operatorname{Sing}_k\Theta\ge \ell\}.
$$

The following result of Koll\'ar's gives a bound on the dimension of $\operatorname{Sing}_k\Theta$.

\begin{teo}[Koll\'ar \cite{ko}] The pair $(A,\Theta)\in \mathcal  A_g$ is log-canonical.  In particular,
$\dim \operatorname{Sing}_k\Theta\le g-k$.
\end{teo}
In other words, $N_\ell^k=\emptyset$ if 
$\ell>g-k$.
This theorem was generalized by Ein-Lazarsfeld \cite{el} who show that for a divisor $D\in |m\Theta|$, $\frac{1}{m}D$ is log-canonical, so that 
$\dim \operatorname{Sing}_{mk}D\le g-k$.  In fact Hacon \cite{hacon} showed that if 
$\lfloor \frac{1}{m}D\rfloor =0$ then $\frac{1}{m}D$  is log-terminal, so that in this case, $\dim \operatorname{Sing}_{mk}D<g-k$.
It follows in particular from Koll\'ar's theorem that for a  point $x\in \Theta$, $\operatorname{mult}_x\Theta\le  g$.
In the case that $\operatorname{Sing}_g\Theta\ne \emptyset$, Smith-Varley \cite{svmultg} show that $(A,\Theta)$ is the product of $g$ elliptic curves.  This result was generalized by Ein-Lazarsfeld:

\begin{teo}[Ein-Lazarsfeld \cite{el}]
If $(A,\Theta)$ is indecomposable, then $\Theta$ is normal, and has rational singularities.
In addition,  
$$\dim \operatorname{Sing}_k\Theta=g-k$$
if and only if 
$(A,\Theta)$ is the product of $k$ ppavs.
\end{teo}

In particular, this implies that if $(A,\Theta)$ is indecomposable, then $$\dim \operatorname{Sing}_k \Theta   <g-k.$$ 
Note also that the theorem generalizes the result of Kempf that the theta divisor of a Jacobian is normal and has rational singularities.  On the other hand, Martens' theorem implies for a Jacobian   the stronger inequality
 $$\dim \operatorname{Sing}_k\Theta_C\le g-2k+1,$$
 with equality holding only for hyperelliptic Jacobians (\ref{eqnmart}).  
We have seen that for a Prym variety $(P,\Xi)$ of dimension $g$, the same inequality holds (Corollary \ref{corprym}).
For $g\le 3$ the two bounds agree.  It was shown by the author that in fact the stronger bound holds for all indecomposable ppavs of dimension $g=4,5$.  This is equivalent to the following theorem.
\begin{teo}[\cite{casaschot}]\label{teocasa5}
Suppose $(A,\Theta)\in \mathcal  A_g$ and $g\le 5$.  If $$\dim \operatorname{Sing}_k\Theta=g-k-\delta,$$
 then $(A,\Theta)$ is the product of $k-\delta$ ppavs, so long as $k-\delta\ge 1$.\end{teo}

In light of this, we  pose the following question.

\begin{que}
Is it  true that for  $(A,\Theta)\in \mathcal A_g$ indecomposable, then $$\dim \operatorname{Sing}_k\Theta\le g-2k+1?$$
\end{que}

 As a consequence, of
 Theorems \ref{teocasa5} and \ref{teocasa3},
 we can describe  the loci $N^k_\ell$ in $\mathcal  A_g$ for $g\le 5$.  We write this out for $g=4,5$ in Tables 1 and 2.  We will use the following notation.
For all $r$-tuples $(i_1,\ldots,i_r)$ such that
$i_1+\ldots +i_r=g$
set $\mathcal  A_{i_1,\ldots,i_r}$ to be the sublocus of $\mathcal A_g$ consisting of products of  $r$ ppavs of dimensions $i_1,\ldots,i_r$ respectively.
Similarly, set $\bar{J}_{i_1,\ldots,i_r}=\bar{J}_g\cap \mathcal  A_{i_1,\ldots,i_r}$, and
$\bar{J}^h_{i_1,\ldots,i_r}=\bar{J}^h_g\cap \mathcal  A_{i_1,\ldots,i_r}$, where $\bar{J}^h_g$ is the closure of the hyperelliptic locus.
We will also use the notation  
$\theta_{\textnormal{null}}^g$ to indicate that the locus is to be considered in $\mathcal A_g$.
\begin{table}\label{tab1}
$$
\begin{array}{c||c|c|c}
4&\mathcal  A_{1,1,1,1}&\emptyset&\emptyset\\
\hline
3&\mathcal  A_1\bar{J}^h_3&\mathcal  A_{1,1,2}&\emptyset\\
\hline
2&\theta_{\textnormal{null}}\cup \bar{J}_4& \bar{J}^h_4\cup \mathcal A_{1,3}&\mathcal  A_{1,3}\cup \mathcal  A_{2,2}\\
\hline
\hline
k/\ell&0&1&2
\end{array}
$$
\caption{The loci $N_\ell^k$ in $\mathcal  A_4$.}
\end{table}

\begin{small}
\begin{table}\label{tab2}
$$
\begin{array}{c||c|c|c|c}
5&\mathcal  A_{1,1,1,1,1}&\emptyset&\emptyset &\emptyset\\
\hline
4&\mathcal  A_{1,1}\bar{J}^h_3&\mathcal  A_{1,1,1,2}&\emptyset &\emptyset\\
\hline
3&\bar{I}\cup \bar{J}_{1,4}\cup \mathcal  A_1\theta^4_{\textnormal{null}}& \bar{J}^h_{1,4}\cup \bar{J}^h_{2,3}\cup \mathcal  A_{1,1,3}&\mathcal  A_{1,1,3}\cup \mathcal  A_{1,2,2} &\emptyset\\
\hline
2&\theta_{\textnormal{null}}\cup N_0'& \bar{J}_5\cup \mathcal A_{1,4}\cup A\cup B\cup C& \bar{J}^h_5\cup \mathcal  A_{1,4}\cup \mathcal  A_{2,3} & \mathcal  A_{1,4}\cup \mathcal  A_{2,3} \\
\hline
\hline
k/\ell&0&1&2&3
\end{array}
$$
\caption{The loci $N_\ell^k$ in $\mathcal  A_5$.}
\end{table}
\end{small}

\subsection{Vanishing theta nulls}
Quadric tangent cones to the theta divisor have been a central feature in the study of ppavs, especially in the geometric Torelli theorems for Jacobians and Pryms.    In \cite{gsm}, Grushevsky and Salvati Manni consider the rank of double points on theta divisors in general.
To state their result, we use the following notation.  Let $\theta_{\textnormal{null}}'$ be the closure of the locus of ppavs in $\theta_{\textnormal{null}}$ whose theta divisor has a unique singularity, which is a double point, but not an ordinary double point.
The fact that $\theta_{\textnormal{null}}'\ne \theta_{\textnormal{null}}$ is due to Debarre (cf. Grushevsky-Salvati Manni \cite{gsm}).  The following theorem gives information about the intersection of the two irreducible divisors in the support of $N_0$.

\begin{teo}[Grushevsky-Salvati Manni \cite{gsm,gsm0}] In the notation above
$\theta_{\textnormal{null}}'\subseteq \theta_{\textnormal{null}}\cap N_0'$.
Moreover, in $\mathcal  A_4$ 
this is an equality.
\end{teo}

The proof is via a detailed analysis of the theta function, and actually provides more information about the intersection $\theta_{\textnormal{null}}\cap N_0'$.  We refer the reader to the paper for more details.   It seems that a proof of the special case of the theorem in dimension four may also be possible via Prym varieties.  The main difficulty remains to find the rank of a quadric tangent cone at an RST-exceptional double point (cf. Question \ref{queprym}).

\section{Singularities of generalized theta divisors}

The techniques used to study the theta divisors of ppavs can also be used to study the singularities of generalized theta divisors on moduli spaces of stable vector bundles on curves.
Recall that given a smooth curve $C$ of genus $g\ge 2$, there are moduli spaces $\mathscr U(k,d)$ of stable vector bundles of rank $k$ and degree $d$ on $C$, and  subvarieties 
$W^r_{\mathscr U(k,d)}$, the so called generalized Brill-Noether loci, parameterizing those vector bundles whose space of global sections is at least $r+1$-dimensional; i.e. as a set
$$
W^r_{\mathscr U(k,d)}=\{
E\in \mathscr U(k,d): h^0(E)\ge r+1\}.
$$

A great deal of work has been done studying the geometry of these varieties, especially determining when they are non-empty, and, since they can be described as degeneracy loci, when they have the expected dimension.   For $k\ge 2$, these loci are still not well understood; we refer the reader to Bradlow et al. \cite{newstead} for some recent results, and a reference to the literature on this subject. 
In what follows we will consider the singularities of these loci.
Note that in the case $k=1$, $d=g-1$, and $r=0$, we recover the theta divisor $\Theta\subset \operatorname{\operatorname{Pic}}^{g-1}(C)$.
More generally, when $d=k(g-1)$ and $r=0$, we recover the so-called generalized theta divisors 
$\Theta_k\subset \mathscr U(k,k(g-1))$, whose singularities were studied by Laszlo.

\begin{teo}[Laszlo \cite{l}]
For $E\in \Theta_k
\subset \mathscr U(k,k(g-1))$, 
$$\operatorname{mult}_E\Theta_k=h^0(E).$$
\end{teo}

\begin{rem}
A similar result was obtained by Laszlo \cite{l} for the restriction of this theta divisor to $S\mathscr U(k,L)$, the moduli of vector bundles of rank $k$ and fixed determinant $L\in \operatorname{\operatorname{Pic}}^{k(g-1)}(C)$.
\end{rem}

\begin{rem}
Recently Hitching \cite[Theorem 14]{hitch} has given a set theoretic geometric description of the tangent cone $C_E\Theta_k$ similar to that given by Kempf for theta divisors on Jacobians.
\end{rem}

More generally, given a vector bundle $F\in \mathscr U(k',d')$, one can consider the subvarieties $W^r_F \subseteq \mathscr U(k,d)$ parameterizing those vector bundles which, when tensored by $F$, have at least $r+1$ linearly independent global sections; i.e. as a set
$$
W^r_F=\{E\in \mathscr U(k,d):h^0(E\otimes F)\ge r+1\}.
$$

When $r=0$, and $k'd+kd'=kk'(g-1)$, these are expected to be divisors, and we use the notation $\Theta_F$.
Laszlo proved a partial result in a special case:

\begin{teo}[Laszlo \cite{l}]
Suppose that $r=0$, $k=1$, $d=0$, $k'=2$ and $d'=2g-2$.
In other words, $F\in \mathscr U(2,2g-2)$, and there is an induced map
$
\operatorname{\operatorname{Pic}}^0(C)\stackrel{\otimes F}{\to} \mathscr U(2,2g-2)
$,
so that pulling back the canonical theta divisor $\Theta_2$ on $\mathscr U(2,2(g-1))$, one obtains $\Theta_F\in |2\Theta|$ on $\operatorname{\operatorname{Pic}}^0(C)$.  

For $\xi$ in $\operatorname{\operatorname{Pic}}^0(C)$, if $h^0(F\otimes \xi)\le 2$, then $\operatorname{mult}_\xi\Theta_F>h^0(F\otimes \xi)$ if and only if there exists an extension of line bundles
$$
0\to A\to F\otimes \xi \to B \to 0
$$
such that $h^0(A)+h^1(B)>h^0(E\otimes \xi)$. 
Moreover, for any curve of genus at least two, there exist vector bundles $F\in \mathscr U(2,2g-2)$, and line bundles $\xi\in \operatorname{\operatorname{Pic}}^0(C)$ such that $\operatorname{mult}_\xi\Theta_F>h^0(F\otimes \xi)$.
\end{teo}

We point out that in contrast to the last statement of the theorem, a result of Teixidor i Bigas \cite{tb} shows that for a general stable rank $2$ degree $2(g-1)$ vector bundle $F$, $\operatorname{mult}_\xi \Theta_F=h^0(F\otimes \xi)$ for all $\xi \in \operatorname{\operatorname{Pic}}^0(C)$.
It seems likely that one can extend the results  cited above.  This is work in progress, joint with Montserrat 
Teixidor i Bigas.  We end by posing some basic questions to which we hope to give partial answers.
\begin{que}Fix $k,d$.  Suppose that 
 $F\in \mathscr U(k',d')$,  $kd'+k'd=kk'(g-1)$, and 
 $\Theta_F=W^0_F$ is a divisor.
When is it true that $\operatorname{mult}_E\Theta_F=h^0(E\otimes F)$?  In the case that
$\operatorname{mult}_E\Theta_F>h^0(E\otimes F)$,  what is the precise multiplicity?
\end{que}

\bibliography{athens}
\end{document}